\documentclass[12pt,oneside,english,jou]{amsart}
\usepackage[T1]{fontenc}
\usepackage[latin9]{inputenc}
\usepackage{verbatim}
\usepackage{amsthm}
\usepackage{amstext}
\usepackage{amssymb}
\usepackage[all]{xy}

\makeatletter
\numberwithin{equation}{section}
\numberwithin{figure}{section}
  \theoremstyle{plain}
  \newtheorem*{thm*}{Theorem}
\theoremstyle{plain}
\newtheorem{thm}{Theorem}[section]
  \theoremstyle{definition}
  \newtheorem{defn}[thm]{Definition}
  \theoremstyle{plain}
  \newtheorem{prop}[thm]{Proposition}
  \theoremstyle{plain}
  \newtheorem{lem}[thm]{Lemma}

\newcommand{\ch}{\mathbf{1}}\newcommand{\al}{\alpha}\newcommand{\T}{\mathcal{T}}\newcommand{\Pcal}{\mathcal{P}}\newcommand{\Ucal}{\mathcal{U}}\newcommand{\D}{\mathcal{D}}\newcommand{\dotcup}{\ensuremath{\mathaccent\cdot\cup}}

\makeatother

\usepackage{babel}

\makeatother

\usepackage{babel}

\makeatother

\usepackage{babel}

\makeatother

\usepackage{babel}

\begin{document}

\title{The universal minimal space for groups of homeomorphisms of h-homogeneous
spaces }

\author{Eli Glasner and Yonatan Gutman}

\address{Eli Glasner, School of Mathematical Sciences, Tel Aviv University,
Ramat Aviv, Tel Aviv 69978, Israel.}

\email{glasner@math.tau.ac.il}

\address{Yonatan Gutman, Laboratoire d'Analyse et de Mathématiques Appliquées,
Université de Marne-la-Vallée, 5 Boulevard Descartes, Cité Descartes
- Champs-sur-Marne, 77454 Marne-la-Vallée cedex 2, France.}

\email{yonatan.gutman@univ-mlv.fr}



\thanks{
%This work was partially done while both authors were visiting the Fields Institute
%in Toronto during the summer of 2010. We thank the Fields Institute for its
%support.
The first named author's research was supported by Grant No 2006119
from the United States-Israel Binational Science Foundation (BSF)}

\dedicatory{To Anatoliy Stepin with great respect.}

\begin{abstract}
Let $X$ be a h-homogeneous zero-dimensional compact Hausdorff space,
i.e. $X$ is a Stone dual of a homogeneous Boolean algebra. It is
shown that the universal minimal space $M(G)$ of the topological
group $G={\rm Homeo}(X)$, is the space of maximal chains on $X$
introduced in \cite{Usp00}. If $X$ is metrizable then clearly $X$
is homeomorphic to the Cantor set and the result was already known
(see \cite{GW03}). However many new examples arise for non-metrizable
spaces. These include, among others, the generalized Cantor sets $X=\{0,1\}^{\kappa}$
for non-countable cardinals $\kappa$, and the \emph{corona} or \emph{remainder}
of $\omega$, $X=\beta\omega\setminus\omega$, where $\beta\omega$
denotes the Stone-$\check{C}$ech compactification of the natural numbers.
\end{abstract}

\maketitle
\pagestyle{plain}

\section{Introduction}

The existence and uniqueness of a universal minimal $G$ dynamical
system, corresponding to a topological group $G$, is due to Ellis
(see \cite{E69}, for a new short proof see \cite{GL11}). He also
showed that for a discrete infinite $G$ this space is never metrizable,
% rrr
and the latter statement was generalized to the locally compact non-compact 
case by Kechris, Pestov and Todorcevic in the appendix to their paper 
\cite{KPT05}.
%This latter result was extended to general locally compact (non-compact)
%groups by Veech \cite{V77}. 
For Polish groups this is no longer the
case and we have such groups with $M(G)$ being trivial (groups with
the fixed point property or extremely amenable groups) and groups
with metrizable, easy to compute $M(G)$, like $M(G)=S^{1}$ for the
group $G={\rm Homeo}_{+}(S^{1})$ (\cite{P98}) and $M(G)=LO(\omega)$,
the space of linear orders on a countable set, for $S_{\infty}(\omega)$
(\cite{GW02}).

Following Pestov's work Uspenskij has shown in \cite{Usp00} that
the action of a topological group $G$ on its universal minimal system
$M(G)$ (with ${\rm card}\, M(G)\ge3$) is never $3$-transitive so
that, e.g., for manifolds $X$ of dimension $>1$ as well as for $X=Q$,
the Hilbert cube, and $X=K$, the Cantor set, $M(G)$ can not coincide
with $X$. 
% rrr
Uspenskij proved his theorem by introducing the space of
maximal chains $\Phi(X)$ associated to a compact space $X$. In \cite{GW03}
the authors then showed that for $X$ the Cantor set and $G={\rm Homeo}(X)$,
in fact, $M(G)=\Phi$. It turns out that this group $G$ is a closed
subgroup of $S_{\infty}(\omega)$ and in \cite{KPT05} Kechris, Pestov
and Todorcevic unified and extended these earlier results and carried
out a systematic study of the spaces $M(G)$ for many interesting
closed subgroups of $S_{\infty}$.

In the present work we go back to \cite{GW03} and generalize it in
another direction. We consider the class of h-homogeneous spaces $X$
and show that for every space in this class the universal minimal
space $M(G)$ of the topological group $G={\rm Homeo}(X)$ is again
Uspenskij's space of maximal chains on $X$. If $X$ is metrizable
then clearly $X$ is homeomorphic to the Cantor set and the result
of \cite{GW03} is retrieved (although even in this case our proof
is new, as we make no use of a fixed point theorem). However, many
new examples arise when one considers non-metrizable spaces. These
include, among others, the generalized Cantor sets $X=\{0,1\}^{\kappa}$
for non-countable cardinals $\kappa$, and the widely studied \textbf{corona}
or \textbf{remainder} \textbf{of $\omega$}, $X=\beta\omega\setminus\omega$,
where $\beta\omega$ denotes the Stone-$\check{C}$ech compactification
of the natural numbers. As in \cite{GW03} the main combinatorial
tool we apply is the dual Ramsey theorem.

\subsection{H-homogeneous spaces and homogeneous Boolean algebras}

The following definitions 
% rrr
are well known (see e.g. \cite{HNV04} Section
H-4):
\begin{enumerate}
\item A zero-dimensional compact Hausdorff topological space $X$ is called
\textbf{h-homogeneous} if every non-empty clopen subset of $X$ is
homeomorphic to the entire space X. %\end{defn}
%A well know definition in the theory of Boolean algebras is the following
%(see \cite{HNV04} Section H-4):
%\begin{defn}

\item A Boolean algebra $B$ is called \textbf{homogeneous} if for any nonzero
element $a$ of $B$ the relative algebra $B|a=\{x\in B:\, x\leq a)$
is isomorphic to $B$.
\end{enumerate}
Using Stone's Duality Theorem (see \cite{BS81} IV$\S4$) a zero-dimensional
compact Hausdorff h-homogeneous space $X$ is the Stone dual of a
homogeneous Boolean Algebra, i.e. any such space is realized %achieved
as the space of ultrafilters $B^{\ast}$ over a homogeneous Boolean
algebra $B$ equipped with the topology given by the base $N_{a}=\{U\in B^{\ast}:\, a\in U\}$,
$a\in B$. Here are some examples of h-homogeneous spaces (see \cite{SR89}):
\begin{enumerate}
\item The countable atomless Boolean algebra is homogeneous. It corresponds
by Stone duality to the Cantor space $K=\{0,1\}^{\mathbb{N}}$.
\item Every infinite free Boolean algebra is homogeneous. These Boolean
algebras correspond by Stone duality to the generalized Cantor spaces,
$\{0,1\}^{\kappa}$ , for infinite cardinals $\kappa$
\item Let $P(\omega)$ be the Boolean algebra of all subsets of $\omega$
(the first infinite cardinal) and let $fin\subset P(\omega)$ be the
ideal comprising the finite subsets of $\omega$. Define the equivalence
relations $A\sim_{fin}B$, $A,B\in P(\omega)$, if and only if $A\triangle B$
is in $fin$. The quotient Boolean algebra $P(\omega)/fin$ is homogeneous.
This Boolean algebra corresponds by Stone duality to the \textbf{corona}
$\omega^{\ast}=\beta\omega\setminus\omega$, where $\beta\omega$
denotes the Stone-$\check{C}$ech compactification of $\omega$. % What about Parovicenko's theorem ?

\item A topological space $X$ is called a \textbf{Parovi\v{c}enko space}
if:

\begin{enumerate}
\item $X$ is a zero-dimensional compact space without isolated points and
with weight $\bold{c}$,
\item every two disjoint open $F_{\sigma}$ subsets in $X$ have disjoint
closures, and
\item every non-empty $G_{\delta}$ subset of $X$ has non-empty interior.
\end{enumerate}
Under CH Parovi\v{c}enko proved that every Parovi\v{c}enko space is
homeomorphic to $\omega^{*}$ (\cite{P63}).

In \cite{vDvM78} van Douwen and van Mill show that under $\neg$
CH, there are two non-homeomorphic Parovi\v{c}enko spaces. Their second
example of a Parovi\v{c}enko space is the corona $X=\beta Y\setminus Y$,
where $Y$ is the $\sigma$-compact space $\omega\times\{0,1\}^{{\bf c}}$.
It is not hard to see that in $Y$ the clopen sets are of the form
$L=\bigcup_{a\in A}\{a\}\times C_{a}$ for some $A\subset\omega$,
where for all $a\in A$, $C_{a}$ is non-empty and clopen. If $|A|=\infty$
then $L\cong Y$ and if $|A|<\infty$ then $Cl_{\beta Y}(L)\subset Y$.
These facts imply in a 
% rrr
straightforward manner that $X$ is h-homogeneous.
In \cite{vDvM78} it is pointed out that under MA $X$ is not
homeomorphic to $\omega^{\ast}$. Thus under $\neg$ CH+MA, this example
provides another weight ${\bf {c}}$ h-homogeneous space.

\item Let $\kappa$ be a cardinal. By a well-known theorem of Kripke (\cite{K67})
there is a homogeneous countably generated complete Boolean algebra,
the so called \textbf{collapsing algebra} $C(\kappa)$ such that if
$A$ is a Boolean algebra with a dense subset of power at most $\kappa$,
then there is a complete embedding of $A$ in $C(\kappa)$.

% new

\item It is not hard to check that the product of any number of h-homogeneous
spaces is again h-homogeneous.
\end{enumerate}

\subsection{The universal minimal space}

A compact Hausdorff $G$-space $X$ is said to be \textbf{minimal}
if $X$ and $\emptyset$ are the only $G$-invariant closed subsets
of $X$. By Zorn's lemma each $G$-space contains a minimal $G$-subspace.
These minimal objects are in some sense the most basic ones in the
category of $G$-spaces. For various topological groups $G$ they
have been the object of intensive study. Given a topological group
$G$ one is naturally interested in describing all of them up to isomorphism.
Such a description is given (albeit in a very weak sense) by the following
construction: as was mentioned in the introduction one can show there
exists a minimal $G$-space $M(G)$ unique up to isomorphism such
that if $X$ is a minimal $G$-space then $X$ is a factor of $M(G)$,
i.e., there is a continuous $G$-equivariant mapping from $M(G)$
onto $X$. $M(G)$ is called the \textbf{universal minimal $G$-space}.
% (for existence and uniqueness see \cite{Usp00}).
Usually this minimal universal space is huge and an explicit description
of it is hard to come by. %For some groups the space itself is complicated, e.g. by a known theorem
%the universal minimal flow of a non-compact locally compact group
%is non-metrizable (see \cite[Theorem A2.2]{KPT05}).

\subsection{The space of maximal chains}

Let $K$ be a compact Hausdorff space. We denote by ${\rm Exp}(K)$
the space of closed subsets of $K$ equipped with the Vietoris topology.
A subset $C\subset{\rm Exp}(K)$ is a \textbf{chain} in ${\rm Exp}(K)$
if for any $E,F\in C$ either $E\subset F$ or $F\subset E$. A chain
is \textbf{maximal} if it is maximal with respect to the inclusion
relation. One verifies easily that a maximal chain in ${\rm Exp}(K)$
is a closed subset of ${\rm Exp}(K)$, and that $\Phi(K)$, the space
of all maximal chains in ${\rm Exp}(K)$, is a closed subset of ${\rm Exp}({\rm Exp}(K))$,
i.e. $\Phi(K)\subset{\rm Exp}({\rm Exp}(K))$ is a compact space.
Note that a $G$-action on $K$ naturally induces a $G$-action on
${\rm Exp}(K)$ and $\Phi(K)$. This is true in particular for $K=M(G)$.
As the $G$-space $\Phi(M(G))$ contains a minimal subsystem it follows
that there exists an injective continuous $G$-equivariant mapping
$f:M(G)\rightarrow\Phi(M(G))$. By investigating this mapping Uspenskij
in \cite{Usp00} showed that for every topological group $G$, the
action of $G$ on the universal minimal space $M(G)$ is not $3$-transitive.
%, i.e.,
%there exist triples $(a_{1},a_{2},a_{3})$ and $(b_{1},b_{2},b_{3})$
%of distinct points of $M(G)$ such that no $g\in G$ satisfies $g(a_{i})=b_{i}$
%for $i=1,2,3$.
As a direct consequence of this theorem only rarely the natural action
of the group $G={\rm Homeo}(K)$ on the compact space $K$ coincides
with the universal minimal $G$-action (as is the case for $X=S^{1}$).
In \cite{Gminimal} it was shown that for $G={\rm Homeo}(X)$, where
$X$ belongs to a large family of spaces that contains in particular
the Hilbert cube, the action of $G$ on the universal minimal space
$M(G)$ is not $1$-transitive.

It is easy to see that every $c\in\Phi(K)$ has a first element $F$
which is necessarily of the form $F=\{x\}$. Moreover, calling $x\triangleq r(c)$
the \textbf{root\/} of the chain $c$, it is clear that the map $\pi:\Phi(K)\to K$,
sending a chain to its root, is a homomorphism of dynamical systems.

\subsection{The main result}

In \cite{GW03} it was shown that the universal minimal space of the
group of homeomorphisms of the Cantor set, equipped with the compact-open
topology, is the space of maximal chains over the Cantor set. Our
goal is to prove the following generalization:
\begin{thm*}
Let $X$ be a h-homogeneous zero-dimensional compact Hausdorff
topological space. Let $G={\rm Homeo}(X)$ equipped with the compact-open
topology, then $M(G)=\Phi(X)$, the space of maximal chains on $X$.
\end{thm*}

\subsection{Acknowledgements}

Part of this work was conducted during the Thematic Program on Asymptotic
Geometric Analysis at the Fields Institute. The authors would like
to thank the Fields Institute for its hospitality. The second author
would like to thank Dana Bartosova and  Stevo Todorcevic for a helpful discussion regarding h-homogeneous spaces and homogeneous Boolean algebras.
Finally, thanks are due to the referee for several helpful remarks.

\section{Preliminaries}

\subsection{Clopen covers}

Let $X$ be a zero-dimensional compact Hausdorff space. Denote by
$\D$ $(\tilde{\D})$ the directed set (semilattice) consisting of
all finite ordered (unordered) clopen partitions of $X$ which are
necessarily of the form $\alpha=(A_{1},A_{2}\ldots,A_{m})$ $(\tilde{\alpha}=\{A_{1},A_{2}\ldots,A_{m}\})$,
where $\dotcup_{i=1}^{n}A_{i}=X$ (disjoint union). The relation is
given by refinement: $\alpha\preceq\beta$ $(\tilde{\alpha}\preceq\tilde{\beta})$
iff for any $B\in\beta$ $(B\in\tilde{\beta})$, there is $A\in\alpha$
$(A\in\tilde{\alpha})$ so that $B\subset A$. The join (least upper
bound) of $\alpha$ and $\beta$, $\alpha\vee\beta=\{A\cap B:\: A\in\alpha,\, B\in\beta\}$,
where the ordering of indices is given by the lexicographical order
on the indices of $\alpha$ and $\beta$ $(\tilde{\alpha}\vee\tilde{\beta}=\{A\cap B:\: A\in\tilde{\alpha},\, B\in\tilde{\beta}\})$.
It is convenient to introduce the notations $\D_{k}=\{\alpha\in\D : \:|\alpha|=k\}$
and $\tilde{\D}_{k}=\{\alpha\in\D : \:|\tilde{\alpha}|=k\}$. We denote
the natural map $(A_{1},A_{2}\ldots,A_{m})\mapsto\{A_{1},A_{2}\ldots,A_{m}\}$
by $\tilde{t}:\D\rightarrow\tilde{\D}$. There is a natural $G$-action
on $\D$ $(\tilde{\D})$ given by $g(A_{1},A_{2}\ldots,A_{m})=(g(A_{1}),g(A_{2})\ldots,g(A_{m}))$
($g\{A_{1},A_{2}\ldots,A_{m}\}=\{g(A_{1}),g(A_{2})\ldots,g(A_{m})\}$).
Let $S_{k}$ denote the group of permutations of $\{1,\ldots,k\}$.
$S_{k}$ acts naturally on $\D_{k}$ by $\sigma(B_{1},B_{2}\ldots,B_{k})=(B_{\sigma(1)},B_{\sigma(2)}\ldots,B_{\sigma(k)})$
for any $\beta=(B_{1},B_{2}\ldots,B_{k})\in\D_{k}$ and $\sigma\in S_{k}$.
This action commutes with the action of $G$, i.e. $\sigma g\beta=g\sigma\beta$
for any $\sigma\in S_{k}$ and $g\in G$. 
% rrr
Notice that one can identify
$\tilde{\D}_{k}=\D_{k}/S_{k}$.

\subsection{Partition Homogeneity}

Let us introduce a new definition:
\begin{defn}
A zero-dimensional compact Hausdorff space $X$ is called \textbf{partition-homogeneous}
if for every two finite ordered clopen partitions of the same cardinality,
$\alpha,\beta\in\D_{m}$, $\alpha=(A_{1},A_{2}\ldots,A_{m})$, $\beta=(B_{1},B_{2}\ldots,B_{m})$
there is $h\in{\rm Homeo}(X)$ such that $hA_{i}=B_{i}$, $i=1,\ldots,k$. \end{defn}
\begin{prop}
\label{pro:h_hom_eq_part_hom}Let $X$ be an infinite zero-dimensional
compact Hausdorff space. $X$ is h-homogeneous iff $X$ is partition-homogeneous. \end{prop}
\begin{proof}
Assume $X$ is h-homogeneous. Let $\alpha,\beta\in\D_{m}$, $\alpha=(A_{1},A_{2}\ldots,A_{m})$,
$\beta=(B_{1},B_{2}\ldots,B_{m})$. Select homeomorphisms $h_{A,i},h_{B,i},$
$i=1,\ldots,m$ with $h_{A,i}:A_{i}\rightarrow X$, $h_{B,i}:B_{i}\rightarrow X$.
Define $g\in{\rm Homeo}(X)$ by $g(x)=h_{B,i}^{-1}\circ h_{A,i}(x)$
for $x\in A_{i}$. Trivially $gA_{i}=B_{i}$. Assume now $X$ is partition-homogeneous.
Let $A\neq X$ be a clopen set in $X$. We distinguish between two
cases:
\begin{enumerate}
\item $A$ is a singleton. As $X$ is partition-homogeneous there exists
$h\in{\rm Aut}(X)$ with $hA=A^{c}$ and $hA^{c}=A.$ We conclude
$X$ is a two point space contradicting the assumption that $X$ is
infinite.
\item % $A\neq\{\ast\}$.
 $A$ is not a singleton. Because $X$ is a compact Hausdorff zero-dimensional
space we can find disjoint clopen sets $A_{1},A_{2}$ such that $A=A_{1}\cup A_{2}$.
Let $h_{1}\in G$ so that $h_{1}A_{1}=A_{1}\cup A^{c}$ and $h_{1}A_{1}^{c}=A_{2}$.
Define the homeomorphism $h:A\rightarrow X$. \[
h(x)=\begin{cases}
h_{1}(x) & x\in A_{1}\\
x & x\in A_{2}\end{cases}\]

\end{enumerate}
\end{proof}

\section{Basic properties of h-homogeneous spaces}

\subsection{Induced orders}

Let $X$ be a compact Hausdorff zero-dimensional h-homogeneous
space and denote $G={\rm Homeo}(X)$. As $X$ is either trivial or infinite, we will assume from now onward, w.l.og. that $X$ is infinite. Let $\upsilon\in\Phi(X)$ and
$D\subset X$ a closed set. Define \[
D_{\upsilon}=\bigcap_{A\in\upsilon:A\cap D\neq\emptyset}A\]
 By maximality of $\upsilon$, one has $D_{\upsilon}\in\upsilon$.
By a standard compactness argument $D_{\upsilon}\cap D\neq\emptyset$
and trivially it is the minimal element of $\upsilon$ that intersects
$D$. Similarly for $D\subset X$ a closed set with $r(\upsilon)\in D$,
define: \[
D^{\upsilon}=\overline{\bigcup_{A\in\upsilon:A\subset D}A}\]
 The maximal element of $\upsilon$ that is contained in $D$.
\begin{defn}
\label{def:Induced order}Let $\upsilon\in\Phi(X)$ and $\tilde{\alpha}=\{A_{1},A_{2}\ldots,A_{m}\}\in\tilde{\D}.$
Define $<_{\upsilon|\tilde{\alpha}}=<_{\upsilon}$, the \textbf{induced
order} on $\tilde{\alpha}$ by $\upsilon$: %\end{defn}
\[
A_{i}<_{\upsilon}A_{j}\Leftrightarrow(A_{i})_{\upsilon}\subseteq(A_{j})_{\upsilon}\]
 Similarly for $\upsilon\in\Phi(X)$ and $\alpha\in\D$, define the
induced order $<_{\upsilon|\alpha}=<_{\upsilon|\tilde{t}(\alpha)}$.
Denote by $t_{\upsilon}^{\ast}:\tilde{\D}\rightarrow\D$ the map $\{A_{1},A_{2}\ldots,A_{m}\}\mapsto(A_{1},A_{2}\ldots,A_{m})$
where $i<j$ if and only if $A_{i}<_{\upsilon|\alpha}A_{j}$. For
$\beta\in\D$, define $t_{\upsilon}^{\ast}(\beta)=t_{\upsilon}^{\ast}(\tilde{t}(\beta))$.
Notice that for all $\sigma\in S_{k}$, $\upsilon\in\Phi(X)$ and
$\beta\in\D$, \[
t_{\upsilon}^{\ast}(\sigma t_{\upsilon}^{\ast}(\beta))=t_{\upsilon}^{\ast}(\beta)\]
 \end{defn}
\begin{lem}
\textup{\label{lem:gt=00003D00003D00003D00003D00003D00003D00003D00003D00003D00003D00003Dtg}$gt_{\upsilon}^{\ast}(\tilde{\beta})=t_{g\upsilon}^{\ast}(g\tilde{\beta})$.} \end{lem}
\begin{proof}
Let $\alpha=(A_{1},A_{2}\ldots,A_{m})=t_{\upsilon}^{\ast}(\tilde{\beta})$.
By definition $i<j$ if and only if $(A_{i})_{\upsilon}\subseteq(A_{j})_{\upsilon}$.
Notice $(gA_{i})_{g\upsilon}=\bigcap_{gA\in g\upsilon:gA\cap gA_{i}\neq\emptyset}gA=g\bigcap_{A\in\upsilon:A\cap A_{i}\neq\emptyset}A=g(A_{i})_{\upsilon}$.
Therefore $i<j$ if and only if $(gA_{i})_{g\upsilon}\subseteq(gA_{j})_{g\upsilon}$,
and we conclude $g\alpha=t_{g\upsilon}^{\ast}(g\tilde{\beta})$. \end{proof}
\begin{prop}
Let $\upsilon\in\Phi(X)$ and $\tilde{\alpha}=\{A_{1},A_{2}\ldots,A_{m}\}\in\tilde{\D}.$
$<_{\upsilon|\tilde{\alpha}}$ is a linear order on %\textup{
$\tilde{\alpha}$. The ordering $A_{i_{1}}<_{\upsilon|\tilde{\alpha}}A_{i_{2}}<_{\upsilon|\tilde{\alpha}}\ldots<_{\upsilon|\tilde{\alpha}}A_{i_{m}}$
is characterized by $(A_{i_{k}})_{\upsilon}\setminus(A_{i_{1}}\cup\ldots\cup A_{i_{k-1}})^{\upsilon}=\{x_{k}\}$
for~$k=1,2,\ldots m$ and 
% rrr
suitable $x_{k}\in A_{k}$. %}
\end{prop}
\begin{proof}
Let $D\subset X$ be clopen so that $r(\upsilon)\in D$, then it is
easy to see that $\upsilon_{|D^{c}}\triangleq\{A\setminus D|\, D^{\upsilon}\subsetneq A\in\upsilon\}$
is a maximal chain in $D^{c}$ and in particular has a root $r(\upsilon_{|D^{c}})=x_{0}\in D^{c}$.
Let $i_{1}$ be such that $r(\upsilon)\in A_{i_{1}}$. Inductively
let $i_{k+1}$ be such that $r(\upsilon_{|(A_{i_{1}}\cup A_{i_{2}}\cup\ldots\cup A_{i_{k}})^{c}})\in A_{i_{k+1}}$.
It is easy to see $A_{i_{1}}<_{\upsilon}A_{i_{2}}<_{\upsilon}\ldots<_{\upsilon}A_{i_{m}}$.
This implies both that $<_{\upsilon}$ is a linear order and $(A_{i_{k}})_{\upsilon}\setminus(A_{i_{1}}\cup\ldots\cup A_{i_{k-1}})^{\upsilon}=\{x_{k}\}$
for some $x_{k}\in A_{k}$, $k=1,2,\ldots m$.
\end{proof}

\subsection{Minimality and proximality of natural actions}

The basis for the Vietoris topology for the compact Hausdorff space
${\rm Exp}(X)$ is given by open sets of the form: \[
\Ucal=\langle{A}_{1},\dots,{A}_{k}\rangle=\{F\in{\rm Exp}(X):\ \,\forall i\, F\cap{A}_{i}\neq\emptyset\textrm{\textrm{ and }}F\subset\bigcup{A_{i}}\}\]
 where ${A}_{i}\subset X$ is clopen. It is easy to see that a basis
of clopen neighborhood of a maximal chain $\upsilon\in\Phi(X)$ is
given by \[
\mathfrak{U_{\alpha}}=\langle\Ucal_{1},\dots,\Ucal_{n}\rangle\]
 where ${\alpha}=(A_{1},A_{2}\ldots,A_{n})\in\D$ and \[
\Ucal_{j}=\langle{A}_{1},\dots,{A}_{j}\rangle,\qquad j=1,2,\dots,n,\]

The following lemma is 
% rrr
straightforward:
\begin{lem}
\label{lem:order_determines_neighborhood} Let $\alpha=(A_{1},A_{2}\ldots,A_{n})\in\D$
and $\upsilon\in\Phi(X)$. Let $<_{\upsilon|\alpha}$ be the induced
order of $\upsilon$ on $\alpha$, then $\upsilon\in\mathfrak{U}{}_{\alpha}$
if and only if %\textup{
$<_{\upsilon}=<$, where $<$ is the usual order on $\{1,2\ldots,n\}$.
In particular $\upsilon\in\mathfrak{U}_{t_{\upsilon}^{\ast}(\alpha)}$.
%}
\end{lem}
\begin{thm}
\label{thm:X_is_minimal&exteremly_proximal} %\end{thm}

\begin{enumerate}
\item The system $(X,G)$ is minimal.
\item The system $(X,G)$ is extremely proximal; i.e. for every closed set
$\emptyset\not=F\subsetneqq X$ there exists a net $\{g_{i}\}_{i\in I}$
in $G$ such that we have $\lim_{i\in I}g_{i}F=\{x_{0}\}$ for some
point $x_{0}\in X$ (see \cite{G74}).
\item The minimal system $(X,G)$ is not isomorphic to the universal minimal
system $(M(G),G)$.
\item $(\Phi(X),G)$ is minimal.
\item \label{enu:Phi_is_proximal}$(\Phi(X),G)$ is proximal.
\end{enumerate}
\end{thm}
\begin{proof}
$ $
\begin{enumerate}
\item Since $X$ is h-homogeneous, then by Proposition \ref{pro:h_hom_eq_part_hom},
$G$ acts transitively on non-trivial (i.e. not $\emptyset,X$) clopen
sets. Since $G$ acts transitively on the above mentioned basis, it
follows that for every $U\in\Ucal$ we have $\cup\{\al(U):\al\in G\}=X$.
This property is equivalent to the minimality of the system $(X,G)$.
\item Fix some $x_{0}$ in $X$ such that $x_{0}\not\in F$. For an arbitrary
basic clopen neighborhood $U=A$ of $x_{0}$ which is disjoint from
$F$ choose $\al_{U}\in G$ such that $\al_{U}(A^{c})=A$. Then $\al$
satisfies $\al_{U}(F)\subset U$. Clearly now $\{\al_{U}:U\ {\text{a neighborhood of}}\allowbreak\ x_{0}\}$
is the required net.
\item %
\begin{comment}
Using Uspenskij's notation, let $\Phi\subset Exp^{2}(X)=Exp(Exp(X))$
be the collection of maximal chains on $X$. Recall that a {\em
chain\/} is a nonempty family $c=\{F_{t}\}$ (with $t$ running
over some parameter set) of closed subsets of $X$ such that for every
$F_{1},F_{2}\in c$ either $F_{1}\subset F_{2}$ or $F_{2}\subset F_{1}$.
A chain $c$ is {\em maximal\/} if it is not properly contained
in another chain. Uspenskij shows that for any dynamical system $(X,G)$
the collection $\Phi$ is a closed invariant subset of the compact
space $Exp^{2}(X)$. It is easy to see that every $c\in\Phi$ has
a first element $F$ which is necessarily of the form $F=\{x\}$.
Moreover, calling $x$ the {\em root\/} of the chain $c$, it
is clear that the map $\rho:\Phi\to X$, sending a chain to its root,
is a homomorphism of dynamical systems.
\end{comment}
% rrr
As the system $(X,G)$ is certainly $3$-transitive this claim
follows from Uspenskij's theorem \cite{Usp00}. 
For completeness we provide a direct proof.
{}Suppose $(X,G)$ is isomorphic to the universal minimal $G$ system.
Let $Y\subset\Phi$ be a minimal subset of $\Phi$. Then, by the coalescence
of the universal minimal system (every $G$-endomorphism $\phi:(M(G),G)\rightarrow(M(G),G)$
(which is necessarily onto) is an isomorphism, see \cite{GL11} and
\cite{Usp00}), the restriction $\pi:Y\to X$, sending a chain to
its root, is an isomorphism. Fix $c_{0}\in Y$ and let $p_{0}\in X$
be its root; i.e. $\pi(c_{0})=p_{0}$. Let $H=\{\al\in G:\al p_{0}=p_{0}\}$,
the stability group of $p_{0}$. Since $\pi$ is an isomorphism we
also have $H=\{\al\in G:\al c_{0}=c_{0}\}$. Choose $F\in c_{0}$
such that $\{p_{0}\}\subsetneqq F\subsetneqq X$ and let $p_{0}\ne a\in F$
(recall $X$ is infinite). Choose a clopen partition of $(P,A,B)$
of $X$ with $B\cap F=\emptyset$, $P\cap F\neq\emptyset$ and $A\cap F\neq\emptyset$.
Using \index{the}the fact that $X$ is partition homogeneous, one
can find $g\in G$ so that $gP=P$, $gA=B$ and $gB=A$. One redefines
$g$ so that $g_{|P}=Id$. As $g(A\cup P)\cap A=\emptyset$, we have
$F\setminus gF\neq\emptyset$. As $gA=B$ we have $gF\setminus F\neq\emptyset$.
Conclude that $F$ and $gF$ are not comparable. On the other hand
$g(p_{0})=p_{0}$ means $g\in H$ whence also $gc_{0}=c_{0}$. In
particular $gF\in c_{0}$ and as $c_{0}$ is a chain one of the inclusions
$F\subset gF$ or $gF\subset F$ must hold. This contradiction shows
that $(X,G)$ cannot be the universal minimal $G$-system.
\item \label{enu:Proof of minimality}

Let $\upsilon',\upsilon\in\Phi(X)$ and $\upsilon'\in\mathfrak{U}_{\alpha}$
for some $\alpha=(A_{1},A_{2}\ldots,A_{n})\in\D$. Let $<_{\upsilon}$
be the induced order of $\upsilon$ on $\alpha$. Let $\sigma\in S_{n}$
be such that for any $i<j$, $A_{\sigma(i)}<_{\upsilon}A_{\sigma(j)}$.
As $X$ is partition homogeneous we can choose $g\in G$ so that $g{A}_{\sigma(i)}={A}_{i}$.
Clearly $g\upsilon\in\mathfrak{\mathfrak{U}_{\alpha}}$.

\item Let $\upsilon_{1},\upsilon_{2}\in\Phi(X)$. Fix some $\upsilon'\in\mathfrak{U_{\alpha}}$
for some $\alpha=(A_{1},A_{2}\ldots,A_{n})\in\D$. Let $<$ be the
usual order on $\{1,2\ldots,n\}$. Inductively we will construct $g\in G$
so that $<_{g\upsilon_{1}|\alpha}=<_{g\upsilon_{2}|\alpha}=<$. Using
Lemma \ref{lem:order_determines_neighborhood}, this implies $g\upsilon_{1}\in\mathfrak{U_{\alpha}}$
and $g\upsilon_{2}\in\mathfrak{U_{\alpha}}$. As $\mathfrak{U_{\alpha}}$
is arbitrary, this establishes proximality. Indeed let $g_{1}\in G$
so that $g_{1}(r(\upsilon_{1})),g_{1}(r(\upsilon_{2}))\in A_{1}$.
Assume we have constructed $g_{k}\in G$. Define $g_{k+1}\in G$ so
that $g_{k+1|A_{1}\cup A_{2}\cup\ldots\cup A_{k}}=g_{k|A_{1}\cup A_{2}\cup\ldots\cup A_{k}}$
and $g_{k+1}(r((g_{k}\upsilon_{1})_{|(A_{i_{1}}\cup A_{i_{2}}\cup\ldots\cup A_{i_{k}})^{c}})$,
$g_{k+1}(r(g_{k}\upsilon_{2})_{|(A_{i_{1}}\cup A_{i_{2}}\cup\ldots\cup A_{i_{k}})^{c}})\in A_{i_{k+1}}$.
It is easy to see that $g=g_{n}$ has the desired properties.
\end{enumerate}
\end{proof}

\section{Calculation of the universal minimal space }

\subsection{Overview}

The goal of this section is to generalize the main theorem of \cite{GW03}:
the universal minimal space of the group of homeomorphisms of the
Cantor set, equipped with the compact-open topology, is the space
of maximal chains over the Cantor set. We prove the following theorem:
\begin{thm}
\label{thm:U_G_for_Homogeneous}Let $X$ be a h-homogeneous
zero-dimensional compact Hausdorff topological space. Let $G={\rm Homeo}(X)$
equipped with the compact-open topology, then $M(G)=\Phi(X)$.
\end{thm}
The proof borrows heavily from the proof in \cite{GW03}. The new
ideas (that build on ideas in \cite{GW03}) are presented in subsections
\ref{sub:Order-topology}, \ref{sub:omega_k}, \ref{sub:Minimal-symbolic-systems}

\subsection{\label{sub:Order-topology}Order topology}

Recall that given a set $Y$ and a linear order $<$ on $Y$ there
is a topology generated by the basis of open intervals $(a,b)=\{y\in Y:\, a<y<b\}$
where $a,b\in Y$ and equality is allowed on the left (right) if $a$
$(b)$ is the smallest (biggest) element of $Y$. This topology is
called the \textbf{order topology} on $(Y,<)$. For more details see
\cite{M75} Section 2.3. One of the most important ingredients in
the proof in \cite{GW03} is the fact that the topology on the cantor
set $K$ is the order topology associated with the natural order $<$
on $K\subset[0,1])$. A natural approach to generalizations of the
result in the case of $X=\omega^{\ast}$ the corona, is to look for
an order that will generate the topology on the corona. However, as
the following proposition shows this is impossible.

\begin{prop}
The topology on $\omega^{\ast}$ is not an order topology. \end{prop}
\begin{proof}
% rrr
Assume for a contradiction that the topology on $\omega^{\ast}$ is
an order topology associated with a linear order $<$. 
As $\omega^{\ast}$ has no isolated points we can find (with no loss
of generality) an increasing bounded sequence of points
$p_1 < p_2 < p_3 < \cdots < b$. By compactness this sequence 
admits a least upper bound $p = l.u.b\ \{p_k: k=1,2,\dots\}$.
It is easy to check that $p = \lim_{k \to \infty}p_k$,
so that the set $\{p_k: k=1,2,\dots\}\cup \{p\}$ is a closed
subset of $\omega^{\ast}$. 
However, it is well known that the remainder $\omega^{\ast}$ has no nontrivial
converging sequences; e.g. one can use the fact that the closure of the set
$\{p_k: k=1,2,\dots\}$, like the closure of any infinite discrete countable set in $\beta\omega$, is homeomorphic to $\beta\omega$
(see e.g. \cite[Theorem 3.6.14]{E}).
%According to
%Section 6.10 of \cite{GJ60}, one can find disjoint clopen subsets
%$A_{n}\subset\omega^{\ast},n=1,2,\ldots,$ so that for any subsequence
%$\{n_{k}\}_{k=1}^{\infty}\subset\mathbb{N}$ and any $p_{k}\in A_{n_{k}}$,
%one has an embedding $i:\beta\{p_{k}\}_{k=1}^{\infty}\hookrightarrow\omega^{\ast}$
%where $i_{|\{p_{k}\}_{k=1}^{\infty}}=Id$. Using the infinite Ramsey
%Theorem one can choose $p_{k}\in A_{n_{k}}$ so that $p_{1}<p_{2}<\ldots$
%or $p_{1}>p_{2}>\ldots$. Assume w.l.o.g the first possibility. Notice
%that an increasing sequence cannot have more than one accumulation
%point contradicting $\beta\{p_{k}\}_{k=1}^{\infty}\cong\beta\mathbb{N}$.
%Indeed if $b$ is an accumulation point of $\{p_{k}\}_{k=1}^{\infty}$
%then one must have $b>p_{k}$ for all $k$. Assume for a contradiction
%there are two accumulation points $b_{2}>b_{1}$. In particular $b_{2}>b_{1}>p_{k}$
%for all $k$ which contradicts $b_{2}$ being an accumulation point.
\end{proof}

\subsection{\label{sub:omega_k}The spaces $\Omega_{k}$ and $\tilde{\Omega}_{k}$
and a cocycle equation}

The following subsection is a generalization of Section 3 of \cite{GW03}.
Fix $\alpha=(A_{1},A_{2}\ldots,A_{k})\in\D_{k}$ and define the clopen
subgroup $H_{\alpha}=\{g\in G:\, gA_{i}=A_{i},\, i=1,\ldots,k\}\subset G$.
Consider the discrete homogeneous space of right cosets $H_{\alpha}\backslash G=\{H_{\alpha}g:\, g\in G\}$.
There is a natural bijection $\phi:H_{\alpha}\backslash G\rightarrow\D_{k}$
given by $\phi(H_{\alpha}g)=g^{-1}\alpha$. Let $\Omega_{k}=\{1,-1\}^{\D_{k}}$
equipped with the product topology. This is a $G$-space under the
action $g\omega(\beta)=\omega(g^{-1}\beta)$ for any $\omega\in\Omega_{k}$,
$\beta\in\D_{k}$ and $g\in G$. %
\begin{comment}
This is a $G$-space under the action $g\omega(H_{\alpha}g')=\omega(H_{\alpha}g'g)$
for every $\omega\in\Omega_{\alpha}$ and $g,g'\in G$.
\end{comment}
{}

Set $\T{}^{k}=\{1,-1\}^{S_{k}}$. We refer to the elements of $\T{}^{k}$
as \textbf{tables}. Denote $\tilde{\Omega}_{k}=(\T{}^{k})^{\tilde{\D}_{k}}$
equipped with the product topology. This is a $G$-space under the
action $\cdot:G\times\tilde{\Omega}_{k}\rightarrow\tilde{\Omega}_{k}$
given by $g\cdot\tilde{\omega}(\tilde{\beta})(\sigma)=\tilde{\omega}(g^{-1}\tilde{\beta})(\sigma)$
for any $\omega\in\Omega_{k}$, $\tilde{\beta}\in\tilde{\Omega}_{k}$
and $g\in G$. %
\begin{comment}
Given $c\in\Phi(X)$ and $\tilde{\beta}=\{B_{1},B_{2}\ldots,B_{k}\}\in\tilde{\D}_{k}$
define $\tilde{\beta}^{\ast}=(B_{\sigma(1)},B_{\sigma(2)}\ldots,B_{\sigma(k)})\in\D_{k}$,
where $\sigma\in S_{k}$ is the unique permutation such that $i<j\Leftrightarrow\sigma(i)<_{c|}\sigma(j)$
\end{comment}
{}

There is a natural family of homeomorphisms $\pi_{c}:\Omega_{k}\rightarrow\tilde{\Omega}_{k}$,
$c\in\Phi(X)$ given by $\omega\mapsto\tilde{\omega}^{c}$, (also
denoted $\tilde{\omega}$ when no confusion arises) where for $\tilde{\beta}=\{B_{1},B_{2}\ldots,B_{k}\}\in\tilde{D}_{k}$
and $\sigma\in S_{k}$, $\tilde{\omega}(\tilde{\beta})(\sigma)=\omega(\sigma^{-1}t_{c}^{\ast}(\tilde{\beta}))$
($t_{c}^{\ast}(\cdot)$ is defined after Definition \ref{def:Induced order}).
In order for $\pi_{c}$ to be a $G$-homeomorphism we need to equip
$\tilde{\Omega}_{k}$ with a different $G$-action than the natural
$G$-action mentioned above. %
\begin{comment}
(notice this action depends on $c)$
\end{comment}
{}%%%%
Namely {}$\bullet_{c}:G\times\tilde{\Omega}_{k}\rightarrow\tilde{\Omega}_{k}$,
is defined by

\[
g\bullet_{c}\tilde{\omega}(\tilde{\beta})(\sigma)=\tilde{\omega}(g^{-1}\tilde{\beta})(\rho_{c}(g,\tilde{\beta})\sigma)=\omega(\sigma^{-1}\rho_{c}(g,\tilde{\beta}){}^{-1}t_{c}^{\ast}(g^{-1}\tilde{\beta}))\]

where $\rho_{c}:G\times\tilde{\Omega}_{k}\rightarrow S_{k}$ is defined
uniquely by the equation: \[
\rho_{c}(g,\tilde{\beta}){}^{-1}t_{c}^{\ast}(g^{-1}\tilde{\beta})=g^{-1}t_{c}^{\ast}(\tilde{\beta})\]
 As $g\bullet_{c}\tilde{\omega}(\tilde{\beta})(\sigma)=\omega(\sigma^{-1}g^{-1}t_{c}^{\ast}(\tilde{\beta}))$,
we have the equality $g\bullet_{c}\tilde{\omega}(\tilde{\beta})(\sigma)=\widetilde{g\omega}(\tilde{\beta})(\sigma)$
which makes $\pi_{c}:(G,\Omega_{k})\rightarrow(G,\tilde{\Omega}_{k})$
a $G$-homeomorphism (and formally proves $g\bullet_{c}$ is indeed
a $G$-action).
\begin{lem}
\textup{\label{lem:cocycle}$\rho_{c}:G\times\tilde{\Omega}_{k}\rightarrow S_{k}$
obeys the following } \textbf{cocyle}\textup{ equation:} \[
\rho_{c}(gh,\tilde{\beta})=\rho_{c}(g,\tilde{\beta})\rho_{c}(h,g^{-1}\tilde{\beta})\]
 \end{lem}
\begin{proof}
By definition we have $gh\bullet_{c}\tilde{\omega}(\tilde{\beta})(\sigma)=\widetilde{gh\omega}(\tilde{\beta})(\sigma)=g\bullet_{c}\widetilde{h\omega}(\tilde{\beta})(\sigma)$.
Notice \[
gh\bullet_{c}\tilde{\omega}(\tilde{\beta})(\sigma)=\omega(\sigma^{-1}\rho_{c}(gh,\tilde{\beta}){}^{-1}t_{c}^{\ast}(h^{-1}g^{-1}\tilde{\beta})),\]
 whereas \[
g\bullet_{c}\widetilde{h\omega}(\tilde{\beta})(\sigma)=h\omega(\sigma^{-1}\rho_{c}(g,\tilde{\beta}){}^{-1}t_{c}^{\ast}(g^{-1}\tilde{\beta}))=\omega(\sigma^{-1}h^{-1}\rho_{c}(g,\tilde{\beta}){}^{-1}t_{c}^{\ast}(g^{-1}\tilde{\beta})).\]
 This implies \[
\rho_{c}(gh,\tilde{\beta}){}^{-1}t_{c}^{\ast}(h^{-1}g^{-1}\tilde{\beta})=h^{-1}\rho_{c}(g,\tilde{\beta}){}^{-1}t_{c}^{\ast}(g^{-1}\tilde{\beta}).\]
 As $\rho_{c}(h,g^{-1}\tilde{\beta}){}^{-1}t_{c}^{\ast}(h^{-1}g^{-1}\tilde{\beta})=h^{-1}t_{c}^{\ast}(g^{-1}\tilde{\beta})$,
we have \[
\rho_{c}(gh,\tilde{\beta}){}^{-1}=\rho_{c}(h,g^{-1}\tilde{\beta}){}^{-1}\rho_{c}(g,\tilde{\beta}){}^{-1}.\]
 Taking the inverses we get $\rho_{c}(gh,\tilde{\beta})=\rho_{c}(g,\tilde{\beta})\rho_{c}(h,g^{-1}\tilde{\beta})$
\end{proof}
Note that in the end of Section 3 of \cite{GW03} it was mistakenly
claimed that $g\bullet_{c_{0}}\tilde{\omega}(\tilde{\beta})(\sigma)=g\cdot\tilde{\omega}(\tilde{\beta})(\sigma)$,
for $c_{0}=\{[0,t]\cap K\}_{t\in[0,1]}$ where $K$, the Cantor set,
is embedded naturally in $[0,1]$.

\subsection{The Dual Ramsey Theorem}

A partition $\gamma=(C_{1},\ldots,C_{k})$ of $\{1,\ldots,s\}$ into
$k$ nonempty sets is\textbf{ naturally ordered }if for any $1\leq i<j\leq k$,
$\min(C_{i})<\min(C_{j}).$ We denote by $\Pi\binom{s}{k}$ the collection
of naturally ordered partitions of $\{1,\ldots,s\}$ into $k$ nonempty
sets.
\begin{defn}
Let $\beta=(B_{1},\ldots,B_{s})\in\Pi\binom{k}{s}$ and $\gamma=(C_{1},\ldots,C_{k})\in\Pi\binom{m}{k}$
define the \textbf{amalgamated partition} $\gamma_{\beta}=(G_{1},\ldots,G_{s})\in\Pi\binom{m}{s}$
by: %\end{defn}
\[
G_{j}=\bigcup_{i\in B_{j}}C_{i}\]

\end{defn}
Notice $\gamma_{\beta}$ is naturally ordered and $(\Pcal_{\gamma})_{\beta}=\Pcal_{\gamma_{\beta}}.$
Similarly for $\alpha=(A_{1},A_{2}\ldots,A_{m})\in\D$ define the
\textbf{amalgamated clopen cover } $\alpha_{\gamma}=({G}_{1},{G}_{1}\ldots,{G}_{k})$,
where ${G}_{j}=\bigcup_{i\in C_{j}}A_{i}$. Notice that % ({\alpha}_\gamma)
${(\alpha}_{\gamma})_{\beta}={\alpha}_{\gamma_{\beta}}.$

We denote by $\tilde{\Pi}\binom{s}{k}$ the collection of unordered
partitions of $\{1,\ldots,s\}$ into $k$ nonempty sets. Notice there
is a natural bijection $\tilde{\Pi}\binom{s}{k}\leftrightarrow\Pi\binom{s}{k}$.
\begin{thm}
\label{thm:Dual_Ramsey}{[}The dual Ramsey Theorem{]} Given positive
integers $k,m,r$ there exists a positive integer $N=DR(k,m,r)$ with
the following property: for any coloring of $\tilde{\Pi}\binom{N}{k}$
by $r$ colors there exists a partition $\alpha=\{A_{1},A_{2},\dots,A_{m}\}\in\tilde{\Pi}\binom{N}{m}$
of $N$ into $m$ non-empty sets such that all the partitions of $N$
into $k$ non-empty sets (i.e. elements of $\tilde{\Pi}\binom{N}{k}$)
whose atoms are measurable with respect to $\alpha$ (i.e. each equivalence
class is a union of elements of $\alpha$) have the same color. \end{thm}
\begin{proof}
This is Corollary 10 of \cite{GR71}.
\end{proof}

\subsection{\label{sub:Minimal-symbolic-systems}Minimal symbolic systems}

In the beginning of Section 4 of \cite{GW03} a family of mappings
$\phi_{T}:(G,\Phi(X))\rightarrow(G,\Omega_{k}),T\in\T^{k}$ are introduced.
We will introduce a generalized family but using a different description.

%The mistake in Section 3 of \cite{GW03} alluded to above affects also
%the definition of the maps
%$\phi_{T}:(G,\Phi(X))\rightarrow(G,\Omega_{k}),T\in\T^{k}$
%at the beginning of Section 4 of \cite{GW03}.
%We will introduce a generalized family by using a different description.

\begin{defn}
%
\begin{comment}
\begin{defn}
Let $\tilde{\beta}\in\tilde{\D}$. Let $c_{1},c_{2}\in\Phi(X)$. Define
the \textbf{$\tilde{\beta}$-ratio} of $c_{1}$ and $c_{2}$, $\sigma_{\tilde{\beta}}(c_{1},c_{2})$
to be unique element $\sigma\in S_{k}$ so that: \[
\sigma t_{c_{1}}^{\ast}(\tilde{\beta})=t_{c_{2}}^{\ast}(\tilde{\beta})\]

\end{defn}

\end{comment}
{}

Let $\beta\in\D$ and $c\in\Phi(X)$, define the \textbf{${\beta}$-ratio}
of $c$, to be the unique element $\theta_{{\beta}}(c)\in S_{k}$
so that: \[
\theta_{{\beta}}(c){\beta}=t_{c}^{\ast}({\beta})\]
 \end{defn}
\begin{lem}
\label{lem:sig_bet_sig_gbet} The following holds:
\begin{enumerate}
\item $\theta_{{\beta}}(c)=\theta_{g{\beta}}(gc)$ for $c\in\Phi(X)$, $g\in G$
and ${\beta}\in\D$.
\item $\theta_{\sigma^{-1}t_{c}^{\ast}(\tilde{\beta})}(c)=\sigma$ for $\sigma\in S_{k}$,
$\tilde{\beta}\in\tilde{\D}$ and $c\in\Phi(X)$.
\end{enumerate}
\end{lem}
\begin{proof}
$ $
\begin{enumerate}
\item By definition $\theta_{g{\beta}}(gc)g{\beta}=t_{gc}^{\ast}(g{\beta})$.
By Lemma \ref{lem:gt=00003D00003D00003D00003D00003D00003D00003D00003D00003D00003D00003Dtg},
$gt_{c}^{\ast}({\beta})=t_{gc}^{\ast}(g{\beta})$ and therefore one
has $\theta_{g{\beta}}(gc)g{\beta}=gt_{c}^{\ast}({\beta})$. As the
$G$ and $S_{k}$ actions commute it implies $\theta_{g{\beta}}(gc){\beta}=t_{c}^{\ast}({\beta})$.
By definition $\theta_{{\beta}}(c){\beta}=t_{c}^{\ast}({\beta})$
and we conclude $\theta_{{\beta}}(c)=\theta_{g{\beta}}(gc)$.
\item $\theta_{\sigma^{-1}t_{c}^{\ast}(\tilde{\beta})}(c)\sigma^{-1}t_{c}^{\ast}(\tilde{\beta})=t_{c}^{\ast}(\sigma^{-1}t_{c}^{\ast}(\tilde{\beta}))$
\end{enumerate}
\end{proof}
Let $T\in\T^{k}$. Define $\phi_{T}:\Phi(X)\rightarrow\Omega_{k}$
by \[
\phi_{T}(c)({\beta})=T(\theta_{{\beta}}(c))\]

\begin{lem}
$\phi_{T}:\Phi(X)\rightarrow\Omega_{k}$ is continuous and $G$-equivariant.
%}
\end{lem}
\begin{proof}
We start by showing that $\phi_{T}$ is continuous. Let $n\in\mathbb{N}$,
$\epsilon_{1},\epsilon_{2},\ldots\epsilon_{n}\in\{\pm1\}$, ${\beta}_{1},{\beta}_{2},\ldots{\beta}_{n}\in\D_{k}$.
Let $V$ be an open set of $\Omega_{k}$ so that $V=\{\omega\in\Omega_{k}:\,\omega({\beta}_{i})=\epsilon_{i}\}$
and assume $V\neq\emptyset$. Let $c_{1}\in\phi_{T}^{-1}(V)$. Denote
$\mathfrak{{U}=\bigcap_{i=1}^{n}}{\mathfrak{U}}_{t_{c_{1}}^{\ast}({\beta})}$.
% \mathfrak{U} here and all through the paragraph ?
By Lemma \ref{lem:order_determines_neighborhood} $c_{1}\in{\mathfrak{U}}$
so $\mathfrak{U}\neq\emptyset$. We claim $\phi_{T}(\mathfrak{U})\subset V$.
Indeed let $c_{2}\in{\mathfrak{U}}$ and fix $i$. By assumption $c_{2}\in{\mathfrak{U}}_{t_{c_{1}}^{\ast}({\beta}_{i})}$.
By Lemma \ref{lem:order_determines_neighborhood} $c_{2}\in{\mathfrak{U}}_{t_{c_{2}}^{\ast}({\beta}_{i})}$.
Conclude $t_{c_{1}}^{\ast}({\beta})=t_{c_{2}}^{\ast}({\beta})$, which
implies $\theta_{{\beta}_{i}}(c_{1})=\theta_{{\beta}_{i}}(c_{2})$.
This in turn implies $\phi_{T}(c_{1})({\beta}_{i})=\phi_{T}(c_{2})({\beta}_{i})=\epsilon_{i}$.

To show $G$-equivariance one has to show $g\phi_{T}(c)({\beta})=\phi_{T}(c)(g^{-1}{\beta})=\phi_{T}(gc)({\beta})$.
By definition $\phi_{T}(c)(g^{-1}{\beta})=T(\theta_{g^{-1}{\beta}}(c))$
whereas $\phi_{T}(gc)({\beta})=T(\theta_{{\beta}}(gc))$.
By Lemma \ref{lem:sig_bet_sig_gbet} $\theta_{{\beta}}(gc)=\theta_{g^{-1}{\beta}}(c)$.
\end{proof}
Let $c_{0}\in\Phi(X)$. We will investigate $\pi_{c_{0}}\circ\phi_{T}$.
By definition $\tilde{\omega}^{c_{0}}(\tilde{\beta})(\sigma)=\omega(\sigma^{-1}t_{c}^{\ast}(\tilde{\beta}))$
and therefore we have $\widetilde{\phi_{T}(c)}^{c_{0}}(\tilde{\beta})(\sigma)=T(\theta_{\sigma^{-1}t_{c_{0}}^{\ast}(\tilde{\beta})}(c))$.
By Lemma \ref{lem:sig_bet_sig_gbet}

\[
\widetilde{\phi_{T}(c_{0})}^{c_{0}}(\tilde{\beta})(\sigma)=\sigma
\]

In particular $\widetilde{\phi_{T}(c_{0})}^{c_{0}}(\tilde{\beta})(\sigma)$
\textit{does not depend} on $\tilde{\beta}$ and we denote it by $\tilde{\omega}_{T}$.

The following theorem is based on Theorem 4.1 of \cite{GW03}:
\begin{thm}
\label{thm:Minimal Factor}Every minimal subsystem of $(G,\Omega_{k})$
is a factor of $(G,\Phi(X))$. \end{thm}
\begin{proof}
Fix a minimal subset $\Sigma\subset\Omega_{k}$. We shall construct
a homomorphism $\phi:(G,\Phi(X))\rightarrow(G,\Sigma)$. Moreover
it will be shown that $\phi=\phi_{T}$ for some $T\in\T^{k}$. Fix
a point $\omega\in\Sigma$ and $c_{0}\in\Phi(X)$. We consider $\tilde{\omega}^{c_{0}}$
as a coloring of elements of $\tilde{\D_{k}}$ by $r=|\T^{k}|$ where
the colors are the tables of $\T^{k}$. For $\tilde{\beta}\in\D_{k}$,
we thus denote by $\tilde{\omega}^{c_{0}}(\tilde{\beta})$ the element
in $\T^{k}$. Let $m\in\mathbb{N}$ and fix ${\alpha}\in\D_{m}$.
Let ${\beta}\in\D$ such that ${\alpha}\preceq{\beta}$, $t_{c_{0}}^{\ast}({\beta})={\beta}$
and $|{\beta}|=N=DR(k,m,r)$ as in Theorem \ref{thm:Dual_Ramsey}.

We define the coloring map to be $f:\Pi{N \choose k}\rightarrow\T^{k}$
where $\gamma\hookrightarrow\tilde{\omega}^{c_{0}}(\tilde{t}({\beta}_{\gamma}))$.
According to Theorem \ref{thm:Dual_Ramsey} there exists $\eta\in\Pi{N \choose m}$
and $T_{{\alpha}}\in\T^{k}$ so that for any $\tau\in\Pi{m \choose k}$,
$f(\eta_{\tau})=T_{{\alpha}}$. Let $g_{{\alpha}}\in G$ be such that
$g_{{\alpha}}^{-1}t_{c_{0}}^{\ast}({\alpha})={\beta}_{\eta}$. Denote
$\tilde{\omega}_{g_{{\alpha}}}^{c_{0}}=g_{{\alpha}}\bullet_{c_{0}}\tilde{\omega}^{c_{0}}$.
Notice \begin{align*}
\tilde{\omega}_{g_{{\alpha}}}^{c_{0}}(\tilde{t}(t_{c_{0}}^{\ast}({\alpha})_{\tau}))(\sigma) & =\omega(\sigma^{-1}g_{{\alpha}}^{-1}(t_{c_{0}}^{\ast}({\alpha})_{\tau}))\\
 & =\omega(\sigma^{-1}(g_{{\alpha}}^{-1}t_{c_{0}}^{\ast}({\alpha}))_{\tau})\\
 & =\omega(\sigma^{-1}({\beta}_{\eta})_{\tau})=\omega(\sigma^{-1}{\beta}_{\eta_{\tau}})\end{align*}
 for any $\tau\in\Pi{m \choose k}$. We also have \[
T_{\alpha}=f(\eta_{\tau})=\tilde{\omega}^{c_{0}}(\tilde{t}({\beta}_{\eta_{\tau}}))(\sigma)=\omega(\sigma^{-1}t_{c_{0}}^{\ast}(\tilde{t}({\beta}_{\eta_{\tau}})))=\omega(\sigma^{-1}{\beta}_{\eta_{\tau}})\]
 %\begin{align*}
%f(\eta_{\tau}) & =\tilde{\omega}^{c_{0}}(\tilde{t}({\beta}_{\eta_{\tau}}))(\sigma)\\
%& =\omega(\sigma^{-1}t_{c_{0}}^{\ast}(\tilde{t}({\beta}_{\eta_{\tau}})))\\
%& =\omega(\sigma^{-1}{\beta}_{\eta_{\tau}})
%\end{align*}
as $t_{c_{0}}^{\ast}({\beta})={\beta}$. Conclude $\tilde{\omega}_{g_{{\alpha}}}^{c_{0}}(\tilde{t}(t_{c_{0}}^{\ast}({\alpha})_{\tau}))=T_{{\alpha}}$.
Let $\tilde{\upsilon}\in\Sigma$ be an accumulation point of the net
$\{\tilde{\omega}_{g_{{\alpha}}}^{c_{0}}\}_{{\alpha}\in\D}$. Let
$\tilde{\xi}_{1},\tilde{\xi}_{2}\in\tilde{\D}_{k}$. Let ${\alpha}$
be a common ordered refinement. By the calculations we have just performed
for any ${\gamma}\succeq{\alpha}$, % ??
$\tilde{\xi}_{1}=\tilde{t}(t_{c_{0}}^{\ast}({\gamma})_{\tau_{1}})$
and $\tilde{\xi}_{2}=\tilde{t}(t_{c_{0}}^{\ast}({\gamma})_{\tau_{2}})$
for some $\tau_{1},\tau_{2}\in\Pi{|{\gamma}| \choose k}$), we have
$\tilde{\omega}_{g_{{\gamma}}}^{c_{0}}(\tilde{\xi}_{1})=\tilde{\omega}_{g_{{\gamma}}}^{c_{0}}(\tilde{\xi}_{2})$.
This implies there exists $T\in\T^{k}$ such that for any $\tilde{\xi}\in\tilde{\D}_{k}$,
$\tilde{\upsilon}(\tilde{\xi})=T$, i.e. $\tilde{\nu}=\tilde{\omega}_{T}$
defined above. We conclude $\Sigma=\phi_{T}(\Phi(X))$.
\end{proof}

\subsection{\label{sub:Universal-minimal-spaces} Calculation of the universal
minimal space}

We now proceed as in \cite{GW03}.
\begin{lem}
\label{lem:X zero implies Homeo zero}If $Y$ is zero-dimensional
compact Hausdorff topological space then the topological group ${\rm Homeo}(Y)$
equipped with the compact-open topology has a clopen basis at the
identity.\end{lem}
\begin{proof}
See the proof of Lemma 3.2 of \cite{MS01}. The clopen basis is given
by $\{H_{{\alpha}}\}_{{\alpha}\in\D}$ where $H_{{\alpha}}$ is defined
in Subsection \ref{sub:omega_k}. \end{proof}
\begin{thm}
\label{thm:U_G is ZD} Let $H$ be a topological group. If the topology
of $H$ admits a basis for neighborhoods at the identity consisting
of clopen subgroups, then $M({H})$ is zero dimensional. \end{thm}
\begin{proof}
This follows from Proposition 3.4 of \cite{P98} where it is shown
that under the same conditions the \textbf{greatest ambit} of $H$
is zero-dimensional.
\end{proof}
We now give the proof of Theorem \ref{thm:U_G_for_Homogeneous}:
\begin{proof}
The proof is a reproduction of the proof appearing in \cite{GW03}
that $M(G)=\Phi(K)$, where $K$ is the Cantor set and $G={\rm Homeo}(K)$
is equipped with the compact-open topology. By Theorem \ref{thm:X_is_minimal&exteremly_proximal}
$(G,\Phi(X))$ is minimal and therefore there is an epimorphism $\pi:(G,M(G))\rightarrow(G,\Phi(X))$.
Fix $c_{0}\in\Phi(X)$ and let $m_{0}\in M(G)$ so that $\pi(m_{0})=c_{0}$.
By Lemma \ref{lem:X zero implies Homeo zero} and Theorem \ref{thm:U_G is ZD}
$M(G)$ is zero-dimensional. Let $D\subset M(G)$ be a clopen subset
and define the continuous function $F_{D}=2\ch_{D}-\ch$, where $\ch_{D}$
is the indicator function of $D$. If $H=\{g\in G:gD=D\}$ then $H$
is a clopen subgroup of $G$ and hence it contains $H_{{\alpha}}$
for some ${\alpha}\in\D_{k}$ for some $k\in\mathbb{N}$ (see
proof of Lemma \ref{lem:X zero implies Homeo zero}). It follows that
the map $\psi_{D}(m)=(F_{D}(gm))_{g\in G}$, $m\in M(G)$ can be defined
as a mapping into $\{1,-1\}^{H_{{\alpha}}\backslash G}=\Omega_{k}$
and thus we have $\psi_{D}:(G,M(G))\to(G,\Omega_{k})$, so that if
we set $Y_{D}=\psi_{D}(M(G))$, the system $(Y_{D},G)$ is a minimal
symbolic subsystem of $\Omega_{k}$. Denote $y{}_{D}=\psi_{D}(m_{0}).$

%
\begin{comment}
We describe the main steps of the proof, pointing out the differences
from the original proof at each step.
\end{comment}
{}Apply Theorem \ref{thm:Minimal Factor} to define a $G$-homomorphism
$\phi_{D}:\Phi\to\Omega_{k}$, with and $y'_{D}=\phi_{D}(c_{0})$.
Given a clopen subset $D\subset M(G)$ consider the following diagram:
\[
\xymatrix{(M(G),m_{0})\ar[d]_{\psi_{D}}\ar[r]^{\pi} & (\Phi,c_{0})\ar[d]^{\phi_{D}}\\
(Y_{D},y_{D}) & (Y_{D},y'_{D})}
\]

The image $(\psi_{D}\times(\phi_{D}\circ\pi))(M(G),m_{0})=(W,(y_{D},y'_{D}))$,
with $W\subset Y_{D}\times Y_{D}$, is a minimal subset of the product
system $(Y_{D}\times Y_{D},G)$. By Theorem \ref{thm:X_is_minimal&exteremly_proximal}(\ref{enu:Phi_is_proximal})
$(Y_{D},G)$ is proximal. Therefore the diagonal $\Delta=\{(y,y):y\in Y_{D}\}$
is the unique minimal subset of the product system and we conclude
that $y_{D}=y'_{D}$, so that the above diagram is replaced by \[
\xymatrix{(M(G),m_{0})\ar[dr]_{\psi_{D}}\ar[rr]^{\pi} &  & (\Phi,c_{0})\ar[dl]^{\phi_{D}}\\
 & (Y_{D},y_{D})}
\]
 Next form the product space \[
\Pi=\prod\{Y_{D}:D\ \text{a clopen subset of}\ M(G)\},\]
 and let $\psi:M(G)\to\Pi$ be the map whose $D$-projection is $\psi_{D}$
(i.e. $(\psi(m))_{D}=\psi_{D}(m)$). We set $Y=\psi(M(G))$ and observe
that, since clearly the maps $\psi_{D}$ separate points on $M(G)$,
the map $\psi:M(G)\to Y$ is an isomorphism, with $\psi(m_{0})=y_{0}$,
where $y_{0}\in Y$ is defined by $(y_{0})_{D}=y_{D}$. Likewise define
$\phi:\Phi(X)\to Y$ by $(\phi(m))_{D}=\phi_{D}(m)$, so that also
$\phi(c_{0})=y_{0}$. These equations force the identity $\psi=\phi\circ\pi$
in the diagram \[
\xymatrix{(M(G),m_{0})\ar[dr]_{\psi}\ar[rr]^{\pi} &  & (\Phi,c_{0})\ar[dl]^{\phi}\\
 & (Y,y_{0})}
\]
 Since $\psi$ is a bijection it follows that so are $\pi$ and $\phi$
and the proof is complete.
\end{proof}
\bibliographystyle{alpha}
%\bibliography{C:/Users/HP/Documents/Documents/Math/Bib/universal_bib}
\bibliography{universal_bib}

\end{document}